\documentclass[preprint]{elsarticle}  
\usepackage[english]{babel}
\usepackage{amsmath,amsfonts,amssymb,color,amsthm}
\usepackage{epstopdf}
\usepackage{natbib}
\usepackage{url}
\usepackage{booktabs}
\usepackage{algorithm}
\usepackage{algpseudocode}
\usepackage{mathdots}
\usepackage{float}
\usepackage{physics}
\usepackage{enumitem}
\usepackage{circuitikz}
\newcommand{\R}{\mathbb R}

\newtheorem{theorem}{Theorem} 
\newtheorem{lemma}{Lemma}
\newdefinition{remark}{Remark}
\newtheorem{definition}{Definition}
\newtheorem{problem}{Problem}

\newtheorem{prop}{Proposition}

\newcommand{\la}[1]{\label{#1}}
\newcommand{\re}[1]{(\ref{#1})}

\newcommand{\bmx}{\begin{bmatrix}}
\newcommand{\emx}{\end{bmatrix}}
\newcommand{\bsm}{\left[\begin{smallmatrix}}
\newcommand{\esm}{\end{smallmatrix}\right]}

\usepackage{multirow,bigdelim}
\usepackage[detect-none]{siunitx}
\usepackage[caption=false]{subfig}
\usepackage{booktabs}
\usepackage{float}
\usepackage{mathtools}

\newcommand{\Blue}[1]{\textcolor{blue}{#1}}
\newcommand{\Z}{\ensuremath{\mathbb{Z}}} 
\newcommand{\B}{\mathcal{B}}    
\newcommand{\calL}{\mathcal{L}} 

\begin{document}
\title{Addition and intersection of\\ linear time-invariant behaviors}
\author[1]{Antonio Fazzi\corref{cor1}} \ead{antonio.fazzi@unipd.it}
\author[2]{Ivan Markovsky} \ead{imarkovsky@cimne.upc.edu}
\cortext[cor1]{Corresponding author}
\address[1]{University of Padova \\ Department of Information Engineering \\ Via Gradenigo 6/b, 35131 Padova, Italy\\[.2cm]}
\address[2]{International Centre for Numerical Methods in Engineering (CIMNE)\\ and Catalan Institution for Research and Advanced Studies (ICREA)\\ Gran Capitàn, 08034 Barcelona, Spain}

\begin{abstract}
  We define and analyze the operations of addition and intersection of linear time-invariant systems in the behavioral setting, where systems are viewed as sets of trajectories rather than input-output maps. The classical definition of addition of input-output systems is addition of the outputs with the inputs being equal. In the behavioral setting, addition of systems is defined as addition of all variables. Intersection of linear time-invariant systems was considered before only for the autonomous case in the context of ``common dynamics'' estimation. We generalize the notion of common dynamics to open systems (systems with inputs) as intersection of behaviors. This is done by proposing trajectory-based definitions. The main results of the paper are 1) characterization of the link between  the complexities (number of inputs and order) of the sum and intersection systems, 2) algorithms for computing their kernel and image representations and 3) a duality property of the two operations. 
  Our approach combines polynomial and numerical linear algebra computations.
\end{abstract}

\begin{keyword}
Behavioral approach \sep addition of behaviors \sep common dynamics.
\end{keyword}

\maketitle

\section{Introduction}
\label{sec-1}


The \textit{behavioral setting} \cite{PW98,W07} is an approach to system theory where systems are defined as sets of trajectories rather than input-output maps. Viewing systems as sets has far-reaching consequences. Most importantly, the system is separated from its numerous representations: a representation is an equation, while the system is the solution set. In systems and control the fundamental object of interest is the solution set and not the equation that defines it. The separation of the notion of a system from the one of a representation also allows one to define equivalence of representations. Naturally, two representations are equivalent when their solution sets are equal.

Another important aspect of the behavioral setting is that the variables of interest are not a priori separated into inputs and outputs. An input/output partitioning is, in general, not unique and may not be a priori given; however, the classical approach imposes a fixed one. This may lead to inconsistencies \cite{W07}.

In this paper, we analyze the basic operations of addition and intersection of linear time-invariant (LTI) systems in the behavioral setting. We give trajectory-based definitions of these operations and algorithms that compute their representations starting from given representations of the original systems. The notion of an addition in the behavioral setting differs from the one in the classical setting. While in the latter only the outputs are added, leaving the inputs the same, in the former, all variables (inputs and outputs) are added. The notion of intersection is not even well-defined in the input-output setting. Only the special case of intersection of autonomous systems is considered in the context of the ``common dynamics'' estimation \cite{Papy06,agcdapp,MLT20,sysid2021}.

\subsection{Literature overview} \label{sec:lit}

The operations of addition and intersection in the behavioral setting appear in the literature. For two-dimensional systems \cite{Valcher1,Valcher2}, conditions that allow representing a system as a direct sum of two systems are given in~\cite{Valcher3}. For one-dimensional systems (the topic of this paper), the sum is used in systems' decompositions, such as
 controllable and autonomous~\cite{PW98},
 stable and unstable~\cite{stableunstable}, and
 zero input and zero initial condition. 
The decomposition into subsystems also appears in the modeling by tearing, zooming and linking~\cite{W07}. 

The intersection operation was used for control in the behavioral setting to restrict the behavior by a controller. Control by intersection of the systems' behaviors is equivalent to what is called \textit{full interconnection} \cite{fullinterconnection}, that is the case when the sets of control variables and to be controlled variables coincide. This is a special case of interconnection \cite{W97} in the behavioral setting, where the constraints are imposed only on a subset of the variables.

The computation of a kernel representation of the sum of two systems is given in \cite[Lemma 2.14]{rev52}. A similar result for the intersection operation is given in Section~\ref{sec-3}. All existing results involving the addition and intersection of LTI systems compute them via their kernel representations.

\subsection{Contribution and organization of the paper}

Although various aspects of the addition and intersection operations for dynamical systems are studied in the literature, a complete treatment of these topics is not available. Also, the computational aspect, i.e., the question of how to find the sum and intersection systems in practice, is missing. These gaps are filled in the present paper. 

First, we give trajectory-based definitions of the addition and intersection operations. Then, we characterize the image and kernel representations of the sum and intersection systems in terms of the image and kernel representations of the original systems. These characterizations lead us to new algorithms for the computation of the image and kernel representations of the sum and intersection systems. The approach that we use is new: we represent polynomial algebra operations by equivalent numerical linear algebra operations based on structured matrices. This methodology is of independent interest and has applications beyond the particular problems we solve in the paper. 
\color{blue}
The algorithms for addition and intersection proposed in the paper are readily implementable in practice. We provide Matlab implementation of the methods in

\vspace{.2cm}
\noindent
\url{https://github.com/fazziant/Other/blob/main/kernel-rep-sum.pdf}.
\vspace{.2cm}
  
  \color{black}
Finally, we show how the sum and intersection systems can be computed directly from observed data from the original systems without resorting to any parametric system representations, such as kernel, image, or state-space representations. This approach is in the spirit of the newly emerged data-driven methods for analysis, control, and signal processing \cite{overview-ddctr,tutorial,identifiability}.

The rest of the paper is organized as follows. Section~\ref{sec-2} reviews results and definitions from behavioral system theory that are used in the paper. We define the trajectory-based operations of addition and intersection of behaviors in Section~\ref{sec-3}, and we propose algorithms for their computation in Section~\ref{sec:alg}. Illustrative examples are given in Section~\ref{sec:ex}.

\section{Notation and preliminaries}
\label{sec-2}

A dynamical system is defined by a triple $(\mathcal{T}, \mathcal{W}, \mathcal{B})$, where $\mathcal{T}$ is the time \Blue{axis}, $\mathcal{W}$ is the variable space (we consider $\mathcal{W} = \R^q$), and $\mathcal{B} \subseteq \mathcal{W}^{\mathcal{T}}$ is the set of admissible trajectories, the \emph{behavior}. In the paper, we focus on discrete-time systems, i.e., $\mathcal{T} = \Z$, and associate the system $(\mathcal{T}, \mathcal{W}, \mathcal{B})$ with its behavior $\mathcal{B}$.

\color{blue}
The system $\mathcal{B}\subseteq (\R^q)^\mathcal{T}$ is \emph{linear} if $\mathcal{B}$ is a subspace of $(\R^q)^{\mathcal{T}}$ and $\mathcal{B}$ is \emph{time-invariant} if it is invariant under the action of the \emph{shift operator}
\[(\sigma w)(t) = w (t+1).\] 
The set of LTI systems with $q$ variables is denoted by $\mathcal{L}^q$. 
We denote with $\mathbf{m}(\mathcal{B})$ / $\mathbf{p}(\mathcal{B})$ the number of inputs / outputs of~$\B$, such that $\mathbf{m}(\mathcal{B}) + \mathbf{p}(\mathcal{B}) = q$.

In the paper, we consider a subclass of the LTI systems, the \emph{finite-dimensional LTI systems}. They admit (vector) difference equation, also called \emph{kernel representation \cite{W86}.}
\color{black}

\begin{lemma}
  \label{lemma:kernel}
  A \Blue{finite-dimensional} LTI system $\mathcal{B}\in\mathcal{L}^q$ has a kernel representation, i.e., there is a matrix polynomial operator
  \[R(z) = R_0  + R_1 z + \cdots + R_\ell z^{\ell} \in \R^{\mathbf{p}(\mathcal{B}) \times q}[z]\]
  such that
  \color{blue}
  \begin{equation}
    \label{eq:kernel}
    \mathcal{B} = \{\ w \in (\mathbb{R}^q)^\mathcal{\mathcal{T}} \ | \ R(\sigma) w = 0 \ \}. 
  \end{equation}
  The smallest $\ell$ for which $\B$ has a kernel representation \re{eq:kernel} with $\deg\,R(z)=\ell$ is invariant of the representation and is called the \textit{lag of $\B$}, denoted by $\mathbf{\ell}(\mathcal{B})$.
\end{lemma}

\color{blue}
A kernel representation is not unique. Given a representation \re{eq:kernel}, an equivalent representation can be obtained by premultiplication of $R(z)$ with a unimodular matrix polynomial $U(z)$, i.e., a square matrix polynomial whose determinant is a nonzero constant. In addition, \re{eq:kernel} may have redundant equations. The representation \re{eq:kernel} is called \emph{minimal} if it has the smallest number of equations. It can be shown (see \cite{W86}) that a minimal kernel representation corresponds to a full row rank matrix polynomial $R(z)$, whose degree is $\ell(\mathcal{B})$. 
Note that in a minimal kernel representation, the number of rows of the polynomial matrix~$R(z)$ is equal to the number of outputs~$\mathbf{p}(\mathcal{B})$.
Every kernel representation can be reduced to a minimal one by a suitable transformation. Thus, in the following, we assume that the kernel representations are minimal.

The variables $w$ of $\B\in\calL^q$ can be partitioned element-wise into inputs $u$ and outputs $y$, i.e., there is a permutation matrix $\Pi$, such that $w = \Pi \bsm u\\ y \esm$ \cite{W07}. This leads to the \emph{input-output representation}
\begin{equation}
  \label{eq:inputoutput}
  \mathcal{B} = \{\, w = \Pi \bsm u\\ y \esm \ | \ Q(\sigma) u = P(\sigma) y \,\},
\end{equation}
where
\begin{equation*}
\begin{aligned}
 & R(z) \Pi =: \bmx Q(z) & -P(z)\emx, \qquad\text{with}\quad  \\ 
 & Q(z) \in \R^{\mathbf{p}(\mathcal{B}) \times \mathbf{m}(\mathcal{B})}[z] 
  \quad\text{and}\quad P(z) \in \R^{\mathbf{p}(\mathcal{B}) \times \mathbf{p}(\mathcal{B})}[z].
\end{aligned}
\end{equation*}
The roots of the polynomial $\det\,P(z)$ are the \emph{poles} of the system (associated with the input/output partitioning $w = \Pi \bsm u\\ y \esm$). 
It can be shown that the degree of $\det\,P(z)$ is invariant of the representation (as long as the representation is minimal) and is, therefore, a property of the system. Indeed, $\det\,P(z)$ is the \textit{order} $\mathbf{n}(\mathcal{B})$ of~$\mathcal{B}$, which is usually defined in terms of a (minimal) state-space representation of~$\B$. 

\color{black}

As all system properties, in the behavioral setting, controllability is also defined in terms of the behavior. 
\begin{definition}
	\label{def:contr}
	A system $\mathcal{B}$ is controllable if for all $w_1, w_2 \in \mathcal{B}$, there exists a $\bar{t} > 0$ and a $w \in \mathcal{B}$ such that
	\begin{equation}
	\notag
	w(t)  = \begin{cases}  w_1(t)  &  \text{for } t < 0 \\
	w_2(t)  & \text{for } t \geq \bar{t}.
	\end{cases}
	\end{equation}
\end{definition}

\color{blue}

The controllability property can be checked in terms of a kernel representation of the system: the system is controllable if and only if the matrix polynomial $R(z)$ in a minimal kernel representation of the system is left prime \cite{cdc19}.

Another representation of an LTI system, used in the paper, is the image representation \cite{PW98}.
 
\begin{lemma}
	\label{lemma:image}
	A controllable system $\mathcal{B} \in \mathcal{L}^q$ has an image representation, i.e., there is a matrix polynomial operator $M(\sigma)$, 
	 such that 
	\begin{equation}\la{eq:image}
   \mathcal{B} = \{\, w \in (\mathbb{R}^q)^{\mathcal{T}} \ | \ w =  M(\sigma) v \,\}.
	\end{equation}
\end{lemma}
\noindent
It holds that $M(z)$ has (column) rank  $\mathbf{m}(\mathcal{B})$ and its row dimension is $q$ \cite{PW98}. 
\color{black}

\begin{definition}
  The rows of the matrix polynomial operator $R(\sigma)$ in \re{eq:kernel} are called  \textit{annihilators} of $\mathcal{B}$. The columns of $M(\sigma)$ in \re{eq:image} are called \textit{generators} of $\mathcal{B}$. 
\end{definition}

Given $\mathcal{B} \in \mathcal{L}^q$, its complexity is defined as the pair  $\big(\mathbf{m}(\B), \mathbf{n}(\B)\big)$. We denote the behavior $\mathcal{B}$ restricted to the interval $[1, L]$ with $\mathcal{B}|_L$. I.e., $\mathcal{B}|_L$ is the set of trajectories truncated to the interval $[1, L]$. The dimension of the restricted behavior $\mathcal{B}|_L$ is
\begin{equation}
\label{eq:formuladim}
\text{dim}\ \mathcal{B}|_L = \textbf{n}(\mathcal{B}) +  L \textbf{m}(\mathcal{B}), \quad\text{for } L \geq  \ell(\mathcal{B}).
\end{equation}

Next, we express $\mathcal{B}|_L$ in terms of the polynomials $R(z)$ and $M(z)$ of a kernel and image representations of the system. For this purpose, we use Hankel and multiplication matrices. Given a time series $w(t) \in \R^q$ of length $T$, the block-Hankel matrix $\mathcal{H}_L(w)$ with $L$ block-rows, where $1 \leq L \leq T$ is defined as
\begin{equation}
  \label{hankel}
  \mathcal{H}_L(w) := \begin{bmatrix}
    w(1) & w(2) & \cdots   & w(T-L+1) \\
    w(2) & w(3) &\cdots   & w(T-L+2) \\
    \vdots & \vdots &  & \vdots \\
    w(L) & w(L+1) &\cdots & w(T)
  \end{bmatrix} \in\R^{qL \times (T-L+1)}.
\end{equation}
In what follows, we will refer to $\mathcal{H}_L(w)$ simply as the Hankel matrix.

Given a polynomial \[r(z) = r_0 + r_1 z + \cdots + r_{\ell} z^{\ell} \in \R^{1\times q}[z]\] of degree $\ell$, the multiplication matrix $\mathcal{T}_L(r)$ with $L$ columns, where $L \geq \ell+1$, is defined as:
\begin{equation}
  \label{toepgen}
  \mathcal{T}_L(r) := \begin{bmatrix}
    r_0 & r_1 &\cdots   & r_{\ell} &   &  &   \\
    & r_0 & r_1 & \cdots & r_{\ell}&   &   \\
    &       & \ddots& \ddots&  & \ddots   \\
    &   &               &r_0 &r_1 & \cdots & r_{\ell} 
  \end{bmatrix} \in\R^{(L - \ell) \times L}.
\end{equation}
For a matrix polynomial $R(z) \in \R^{p\times q}[z]$ we define the (generalized) multiplication matrix $\mathcal{T}_L(R)$ with $L$ columns in terms of the rows $R^1(z),\ldots,R^p(z)$:
\[
\mathcal{T}_L(R) := \text{ker}\,\begin{bmatrix}
	\mathcal{T}_L(R^1) \\ \vdots \\ \mathcal{T}_L(R^p)
	\end{bmatrix}.
\]

\begin{definition}
  \label{persi}
  A time series $u =\big(u(1), u(2), \dots, u(T)\big)$ is persistently exciting of order $L$ if the Hankel matrix $\mathcal{H}_L(u)$  is full row rank.  
\end{definition}

We can now state the connection between a (finite length) behavior and the Hankel matrix built from an observed trajectory $w$ \cite{WRMD}.
\begin{lemma}
  \label{lemma:hankel}
  If $\mathcal{B} \in \mathcal{L}^q$ is controllable, $w \in \mathcal{B}|_T$, and the input component~$u$ of~$w$ is persistently exciting of order $L+\mathbf{n}(\mathcal{B})$ ($L \geq \mathbf{\ell}(\mathcal{B})+1$), then  
  \begin{equation*}
    \mathcal{B}|_L = \text{image}\, \mathcal{H}_L(w).
  \end{equation*}
\end{lemma}  

\color{blue}
Lemma \ref{lemma:hankel} is a classic result known as the \emph{fundamental lemma}. It has been recently shown  \cite{identifiability} that the assumptions of Lemma \ref{lemma:hankel} (controllability, given input/output partitioning, and persistency of excitation of the input) can be replaced by the following condition
\begin{equation}
  \label{eq:inputfromdata}
  \text{rank} \, \mathcal{H}_L(w) =  \textbf{n}(\mathcal{B}) + L \textbf{m}(\mathcal{B}).
\end{equation}
We refer to \re{eq:inputfromdata} as the \emph{generalized persistency of excitation condition}. Observe that the right-hand side of \eqref{eq:inputfromdata} is the same as in \eqref{eq:formuladim}, hence under the generalized persistency of excitation condition, we have that
\begin{equation}
\label{eq:dimB}
 \text{dim}\ \mathcal{B}|_L = \text{rank}\, \mathcal{H}_L(w).
\end{equation}
This allows us to compute the complexity of $\mathcal{B}$ directly from an observed trajectory $w$ by solving a system of linear equations, see Algorithm \ref{alg:comp}.

\begin{algorithm}
  \caption{Computation of system complexity from a trajectory}
  \label{alg:comp}
  \begin{algorithmic}
    \Require a trajectory $w\in(\R^q)^T$ of $\mathcal{B}$
    \State 1: let $L = \lfloor \frac{T+1}{q+1} \rfloor$
    \State 2: compute $r_1 = \text{rank}\, H_L(w)$ and $r_2 = \text{rank}\, H_{L-1}(w)$
    \State 3: solve the system of equations
    \begin{equation}\la{sys}
      \bmx
      L & 1 \\ L-1 & 1
      \emx \bmx m\\ n \emx = \bmx
      r_1 \\ r_2
      \emx
    \end{equation}
    \Ensure data-generating system's complexity $(m, n)$ 
  \end{algorithmic}
\end{algorithm}

\begin{prop}
  Given a trajectory $w \in \mathcal{B}|_T$, such that the generalized persistency of excitation condition \re{eq:inputfromdata} holds for $L:=\lfloor \frac{T+1}{q+1} \rfloor$ and $L\geq\mathbf{\ell}(\B)+1$, Algorithm \ref{alg:comp} computes the complexity of $\mathcal{B}$. 
\end{prop}
\begin{proof}
The matrix in the left-hand side of \re{sys} is nonsingular, so a solution exists and is unique. The fact that $m = \mathbf{m}(\mathcal{B})$ and $n = \mathbf{n}(\mathcal{B})$, as claimed, is a direct consequence of \eqref{eq:inputfromdata}, which holds by the assumptions of the proposition. 
\end{proof}

\color{black}

\begin{lemma} 
	\label{lemma:syl}
	For an LTI system $\mathcal{B}$ with a kernel representation $\mathcal{B} = \text{ker} \, R(\sigma)$, 
	\begin{equation*}
	\mathcal{B}|_L= \text{ker}\,\mathcal{T}_L(R), \quad \text{for } L \geq \mathbf{\ell}(\B)+1.
	\end{equation*}
\end{lemma}
\begin{proof}
Consider a finite trajectory $w\in\mathcal{B}|_T$ of $\B = \ker\,R(\sigma)$. For each row~$R^i(\sigma)$ of $R(\sigma)$, with degree $\ell_i:=\deg\,R^i(\sigma)$, we have 
  \begin{multline}
    \label{eq:proof}
    R_0^i w(t) + R_1^i \sigma  w(t) + \cdots + R_{\ell_i}^i \sigma^{\ell}  w(t) = 0, \\
    \text{for } t=1, \dots, T-\ell_i \text{ and } i = 1,\ldots,\mathbf{p}(\B),
  \end{multline}
Written in matrix form, the system of equations~\eqref{eq:proof} is $\mathcal{T}_L(R) w = 0$.
\end{proof}

Lemmas \ref{lemma:hankel} and \ref{lemma:syl} are useful because they link behaviors, trajectories, representations, and structured matrices. Hence these results connect system theory using the behavioral approach, linear algebra, and matrix computation.
The result in Lemma \ref{lemma:syl} allows the construction of  finite-length trajectories starting from the system kernel representation. \textcolor{blue}{This is used in \cite{distanceECC23} to define a representation-invariant distance between behaviors.}

\begin{remark}
This work deals with deterministic systems and exact (noiseless) data. If the data are affected by noise, the Hankel matrix $H_L(w)$ is full rank for all $L$. A possible approach for dealing with noisy data is to preprocess the data via Hankel low-rank approximation \cite{hankelode,slra} to satisfy the rank condition in \eqref{eq:inputfromdata}.  
\end{remark}

\section{Addition and intersection of behaviors}
\label{sec-3}

We define the operations of addition and intersection of two LTI behaviors by looking at  the relation between their dimensions and the ones of the original systems.  Then, we state how to compute  the representations of the addition and intersection systems starting from the ones of the original systems. These results can be naturally extended to more than two behaviors. 

\begin{definition}
Given two behaviors, $\mathcal{A}$ and $\mathcal{B}$, with the same number of variables,  their sum is naturally defined as the set of the sums of the elements of $\mathcal{A}$ and $\mathcal{B}$, while their intersection  is defined as the set of  elements that belong to both $\mathcal{A}$ and $\mathcal{B}$:
\begin{equation}
\label{addbeh}
\begin{aligned}
\mathcal{B}_+ &= \mathcal{A} + \mathcal{B} := \{ w=a+b\ |\ a \in \mathcal{A}, b \in \mathcal{B} \}. \\
\mathcal{B}_{\cap} &= \mathcal{A} \cap \mathcal{B} := \{ w \ |\  w \in \mathcal{A} \ \text{and} \ w \in \mathcal{B} \}.
\end{aligned}
\end{equation}
\end{definition}
It can be checked that 
if $\mathcal{A}, \mathcal{B} \in \mathcal{L}^q$, then $\mathcal{B}_+, \mathcal{B}_{\cap} \in \mathcal{L}^q$. In this case, 
the following result states the link between   the starting systems' dimensions and the ones of their sum and  intersection.

\begin{lemma}
		\label{comsum}
		Let $\mathcal{A}|_L, \mathcal{B}|_L \in \mathcal{L}^q$ be two behaviors restricted to the interval $[1, L]$, and consider their sum $\mathcal{B}_+|_L$ and their intersection $\mathcal{B}_{\cap}|_L$. The dimensions of the sum and intersection systems are related to the ones of $\mathcal{A}|_L$ and $\mathcal{B}|_L$ as follows
		\begin{equation}
		\label{eq:dimensions}
		\text{dim}\ (\mathcal{B}_+|_L) = 	\text{dim}\ (\mathcal{A}|_L) + 	\text{dim}\ (\mathcal{B}|_L) - 	\text{dim}\ (\mathcal{B}_{\cap}|_L)  \quad \text{for } L \geq \ell(\mathcal{B}_+).
		\end{equation}
\end{lemma}

\begin{proof}
	Equation \eqref{eq:dimensions} is a consequence of \eqref{eq:formuladim}. If we expand all the terms, we get
   	\begin{equation*}
   	\begin{aligned}
   	\text{dim}\ (\mathcal{B}_+|_L) &= \mathbf{n}(\mathcal{B}_+) +  \mathbf{m}(\mathcal{B}_+) L \\
   	\text{dim}\ (\mathcal{A}|_L) &= \mathbf{n}(\mathcal{A}) +  \mathbf{m}(\mathcal{A}) L \\
   	\text{dim}\ (\mathcal{B}|_L) &= \mathbf{n}(\mathcal{B}) +  \mathbf{m}(\mathcal{B}) L \\
   	\text{dim}\ (\mathcal{B}_{\cap}|_L) &= \mathbf{n}(\mathcal{B}_{\cap}) +  \mathbf{m}(\mathcal{B}_{\cap}) L.
   	\end{aligned}
   	\end{equation*}
   	The result follows from the straightforward definitions 
   	\begin{equation}
   	\label{eq:inputssum}
  \begin{aligned}
   \mathbf{n}(\mathcal{B}_+) &=  \mathbf{n}(\mathcal{A}) + \mathbf{n}(\mathcal{B}) - \mathbf{n}(\mathcal{B}_{\cap}) \\ 
    \mathbf{m}(\mathcal{B}_+) &=  \mathbf{m}(\mathcal{A}) + \mathbf{m}(\mathcal{B}) - \mathbf{m}(\mathcal{B}_{\cap}) 
  \end{aligned}
  \end{equation}
\end{proof}

\begin{remark}
  \label{rmk:noout}
  Depending on the number of inputs  $\mathbf{m}(\mathcal{A})$ and $\mathbf{m}(\mathcal{B})$, the systems $\mathcal{B}_+$ and $\mathcal{B}_{\cap}$ may be trivial systems (systems where all the variables are  inputs).
\end{remark}

If the systems $\mathcal{B}_+$ and $\mathcal{B}_{\cap}$ are in $\mathcal{L}^q$, they admit kernel representations, which can be expressed in terms of the kernel representations of $\mathcal{A}$ and $\mathcal{B}$. If these systems are also controllable, they admit  image representations too, that can be characterized in terms of the image representations of $\mathcal{A}$ and $\mathcal{B}$.
\begin{theorem}
	\label{th:dual1}
	Let $\mathcal{A}, \mathcal{B} \in \mathcal{L}^q$. 
The following hold true:
	\begin{enumerate}
		\item \textcolor{blue}{If $\mathcal{A}, \mathcal{B}$ are controllable, let $P_a, P_b$ be their image representations.} An image representation of their sum is given by the union of generators:
		\begin{equation*}
		\mathcal{B}_+ = \mathcal{A} + \mathcal{B} = \text{image}\,  P_+(\sigma) := \text{image}\, \begin{bmatrix} P_a & P_b \end{bmatrix} (\sigma). 
		\end{equation*}
		\item \textcolor{blue}{Let $R_a, R_b$ be the kernel representations of $\mathcal{A}, \mathcal{B}$.} A kernel representation of their intersection is given by the union of annihilators:
		\begin{equation}
		\label{eq:kercap}
	  \mathcal{B}_{\cap} = 	\mathcal{A} \cap \mathcal{B} = \text{ker}\,  R_{\cap}(\sigma) := \text{ker}\,  \begin{bmatrix} R_a \\ R_b \end{bmatrix} (\sigma). 
		\end{equation}
	\end{enumerate}
\end{theorem}

\begin{proof}
	  Consider the first  point. \textcolor{blue}{First of all, the controllability assumption guarantees the existence of the image representations}.  If $w_a = P_a (\sigma) \ell_a$ and  $w_b=P_b (\sigma) \ell_b$ are two trajectories of $\mathcal{A}$ and $\mathcal{B}$, respectively, then
  \[w_a+w_b = \begin{bmatrix} P_a & P_b \end{bmatrix} (\sigma) \begin{bmatrix} \ell_a \\ \ell_b \end{bmatrix} \in \mathcal{A} + \mathcal{B}.\]
  Hence, the image representation of the sum system is obtained by stacking next to each other the two image representations of the starting systems. 
		
  For the second point, let $0 = R_a (\sigma) z = R_b (\sigma) z$ for a certain trajectory $z \in   \mathcal{A} \cap \mathcal{B}$. Then \[0 = \begin{bmatrix} R_a \\ R_b \end{bmatrix} (\sigma) z.\] Therefore it follows the expression for the kernel representation of the intersection system starting from the given kernel representations \eqref{eq:kercap}.  
\end{proof}

Theorem \ref{th:dual1} shows a duality between the addition and the 
intersection of behaviors and the corresponding representations as
union of generators and intersection of  annihilators of the starting behaviors. This means that the representations of the sum and intersection systems can be computed with opposite operations  by switching  annihilators with generators, union with intersection, and row-wise with column-wise operations on some matrices (that are built from the given representations). We  expect that similar 
relations still hold true by reversing  the computations of the sum and intersection systems as 
intersection of annihilators and generators, respectively, as shown in Table \ref{tab:duality}. However, these  computations (intersection of annihilators or generators)
need to be implemented in an algorithm. 
\begin{table}[ht]
	\caption{Duality between addition and intersection behaviors representations and their relation with respect to the starting systems representations. \label{tab:duality}}
	\centering
	\begin{tabular}{l|c|c|}
		& generators $P$ & annihilators $R$\\
		\hline
		$\mathcal{A} + \mathcal{B}$ & $\cup$ & $\cap$\\
		\hline
		$\mathcal{A} \cap \mathcal{B}$ & $\cap$ & $\cup$\\
		\hline
	\end{tabular}
\end{table}

\section{Intersection of annihilators and generators}
\label{sec:alg}
The union of generators and annihilators is easy to be computed. But we may need to find the sum or intersection systems representations by the intersection of annihilators or generators, respectively. 
We describe, in the following, the problem of intersection of annihilators. Then, we can apply the observed duality property to get the \textit{dual} computational algorithm for the intersection of generators. 
\begin{problem}
	Given minimal kernel representations of the two   behaviors $\mathcal{A}, \mathcal{B} \in \mathcal{L}^q$, i.e.,
	\begin{equation*}
	\mathcal{A} = \text{ker} \, R_a(\sigma) \ \  \text{and} \ \ \ \mathcal{B} = \text{ker} \, R_b (\sigma),
	\end{equation*}
	with polynomials
	\begin{equation*}
	\begin{aligned}
	R_a(z) &= R_{a,0} z^0 + R_{a,1} z^1 + \cdots + R_{a,\ell_a} z^{\ell_a} \in \R^{p_a \times q} [z], \\
	R_b(z) &= R_{b,0} z^0 + R_{b,1} z^1 + \cdots + R_{b,\ell_b} z^{\ell_b} \in  \R^{p_b \times q} [z],
	\end{aligned}
	\end{equation*}
	find a kernel representation $R_+(\sigma)$ of the sum $\mathcal{A} + \mathcal{B}$, where
	\begin{equation*}
	R_+(z) = R_0 z^0 + R_1 z^1 + \cdots + R_{\ell} z^{\ell}.
	\end{equation*}
\end{problem}
By solving this problem, we obtain a computational method for $(R_a, R_b) \rightarrow R_+$, which is a direct way to get the kernel representation of the sum from the kernel representations of the starting systems.   The computational algorithm for this problem is  Algorithm \ref{alg:intannihilators}, whose correctness is proved below.

\begin{algorithm}[H]
	\caption{Sum of two behaviors by intersection of annihilators}
	\label{alg:intannihilators}
	\begin{algorithmic}
		\Require $R_a, R_b$ minimal kernel representations of the starting systems
		\Ensure $R_+$ kernel representation of the sum system
		
		\State Compute the number of  outputs of the sum $\mathbf{p}(\mathcal{B_+}) = q - \mathbf{m}(\mathcal{B_+})$
		\State  Build the Sylvester matrix
		\begin{equation}
		\label{eq:Sylv}
		\mathcal{S}_L = \bmx \mathcal{T}_L(R_a) \\ \mathcal{T}_L(R_b) \emx
		\end{equation}
		by choosing $L$ such that the left kernel of $\mathcal{S}_L$ has dimension $\mathbf{p}(\mathcal{B_+})$
		\State Compute the left kernel basis $\begin{bmatrix} Z_a & -Z_b
			\end{bmatrix}$ of the Sylvester matrix $\mathcal{S}_L$
		\State Define $R_+$ as
		   \begin{equation*}
		\begin{bmatrix} R_0 & R_1 & \cdots & R_{\ell} \end{bmatrix} := Z_a
			\mathcal{T}  _L(R_a) = Z_b \mathcal{T}_L(R_b).
			\end{equation*}
	\end{algorithmic}
\end{algorithm}

\begin{prop}
	\label{prop2}
Algorithm 
 \ref{alg:intannihilators} computes the kernel representation $R_+$ of the sum of the two starting  behaviors $\mathcal{A} + \mathcal{B}$.
 \end{prop}
\color{blue}
Before proving Proposition \ref{prop2}, we make some comments on the parameter $L$ that defines the dimension of the Sylvester matrix \eqref{eq:Sylv}. 
	The Sylvester matrix has two multiplication blocks generated by the starting representations,  whose dimensions are $\mathbf{p}(\mathcal{A})(L-\ell_a) \times qL$ and $\mathbf{p}(\mathcal{B})(L-\ell_b) \times qL$, respectively, for a certain $L \geq \max(\ell_a, \ell_b) +1$. If $\mathbf{p}(\mathcal{B_+}) > 0$, we can always choose $L$ such that $S_L$ has a non-trivial left kernel of dimension $\mathbf{p}(\mathcal{B_+})$. This means that the difference between the number of rows and the number of columns of $S_L$ is (at least) $\mathbf{p}(\mathcal{B_+})$, leading to $L = \frac{\mathbf{p}(\mathcal{B_+}) + \ell_a\mathbf{p}(\mathcal{A}) + \ell_b \mathbf{p}(\mathcal{B})}{\mathbf{p}(\mathcal{A}) + \mathbf{p}(\mathcal{B}) - q}$. 
The parameter $L$ defines the number of (block) columns of the two multiplication matrices $ \mathcal{T}_L(R_a)$ and  $\mathcal{T}_L(R_b)$. But these blocks can have different row dimensions, depending on the number of outputs and the lag of the two behaviors.
\color{black}
\begin{proof}
The number of outputs of the sum system is needed to understand the number of rows of the sought (minimal) representation, and it can be computed from the starting systems by \eqref{eq:inputssum}.

 Consider then two trajectories $a \in \mathcal{A} \subset \mathcal{A} + \mathcal{B}$ and $b \in \mathcal{B} \subset \mathcal{A} + \mathcal{B}$.  If the Sylvester matrix $S_L$ has a nontrivial left kernel, the sought  representation $R_+(\sigma)$ satisfies 
 \begin{itemize}
 	\item $R_+(\sigma) a = 0$ $\implies$ $R_+$ is in the row span of  $\mathcal{T}_L(R_a) \implies R_+ = Z_a \mathcal{T}_L(R_a)$; 
 	\item $R_+(\sigma) b = 0$ $\implies$ $R_+$ is in the row span of $\mathcal{T}_L(R_b)  \implies R_+ = Z_b \mathcal{T}_L(R_b)$; 
 \end{itemize}
leading to the equation 
$Z_a \mathcal{T}_L(R_a) =  Z_b \mathcal{T}_L(R_b)$ $\implies$ $\begin{bmatrix} Z_a & -Z_b \end{bmatrix} \mathcal{S}_L = 0$. Observe that, by construction, $\begin{bmatrix} Z_a & -Z_b \end{bmatrix}$ has (at least) $\mathbf{p}(\mathcal{B_+})$ rows. 
We can define the coefficients of $R_+(z)$  as
\begin{equation*}
\begin{bmatrix} R_0 & R_1 & \cdots & R_{\ell} \end{bmatrix} := Z_a
\mathcal{T}  _L(R_a) = Z_b \mathcal{T}_L(R_b).
\end{equation*}
Since $R_+$ is the kernel representation of both the behaviors $\mathcal{A}$ and $\mathcal{B}$, it is also a representation of their sum $\mathcal{A}+\mathcal{B}$. 
\end{proof}

\vspace{.2cm}
	 
 Algorithm \ref{alg:intannihilators} always computes a solution for the kernel representation if the sum system is nontrivial. But, 
if $\mathbf{p}(\mathcal{B_+}) = q -  \mathbf{m}(\mathcal{B_+}) = 0$, the kernel representation is trivial (the behavior has no annihilators different from the zero polynomial). In this case, the Sylvester matrix $\mathcal{S}$  always  has more   columns than rows  for each value of $L$ (or possibly it is square), so the left kernel is trivial (we assume the matrix polynomials $R_a$ and $R_b$ have no common factors, see \cite{matpol}). 
	 
\vspace{.5cm}
Next, we state the \textit{dual} problem for the intersection of generators, where the controllability of the behaviors is assumed to guarantee the existence of the image representation. 
\begin{problem}
  Given  image representations of the two controllable  behaviors $\mathcal{A}, \mathcal{B} \in \mathcal{L}^q$, i.e.,
  \begin{equation*}
    \mathcal{A} = \text{image} \, P_a(\sigma) \ \  \text{and} \ \ \ \mathcal{B} = \text{image} \, P_b (\sigma),
  \end{equation*}
	with polynomials
	\begin{equation*}
	\begin{aligned}
	P_a(z) &= P_{a,0} z^0 + P_{a,1} z^1 + \cdots + P_{a,\ell_a} z^{\ell_a} \in \R^{q \times m_a} [z], \\
	P_b(z) &= P_{b,0} z^0 + P_{b,1} z^1 + \cdots + P_{b,\ell_b} z^{\ell_b} \in  \R^{q \times m_b} [z],
	\end{aligned}
	\end{equation*}
	find an image  representation $P_{\cap}(\sigma)$ of the intersection $\mathcal{A} \cap \mathcal{B}$, where
	\begin{equation*}
	P_{\cap}(z) = P_0 z^0 + P_1 z^1 + \cdots + P_{\ell} z^{\ell}.
	\end{equation*}
\end{problem}
\color{blue}
The computational procedure to approach this problem is \textit{dual} with respect to the previous case and it is illustrated in Algorithm \ref{alg:intgenerators}. 

Because of the \emph{duality}, 
 the multiplication blocks are now stacked in a row, and the (transposed) Sylvester matrix should have a kernel of dimension $\mathbf{m}(\mathcal{B_{\cap}})$. The key point is to observe that the image representation of the intersection lies in the column span of both the starting representations $P_a$ and $P_b$.
\color{black}

\begin{algorithm}
	\caption{Intersection of two behaviors by intersection of generators}
	\label{alg:intgenerators}
	\begin{algorithmic}
		\Require $P_a, P_b$ image representations of the starting systems
		\Ensure $P_{\cap}$ image representation of the intersection system
		
		\State Compute the number of inputs  of the intersection  $\mathbf{m}(\mathcal{B_{\cap}})$
		\State  Build the (transposed) Sylvester matrix
		\[\mathcal{S}_L = \bmx \mathcal{T}_L^T(P_a^T) &  \mathcal{T}_L^T(P_b^T) \emx\]
		by choosing $L$ such that the  kernel of $\mathcal{S}_L$ has dimension $\mathbf{m}(\mathcal{B_{\cap}})$
		\State Compute the  kernel basis $\begin{bmatrix} Z_a \\ -Z_b
		\end{bmatrix}$ of the Sylvester matrix $\mathcal{S}_L$
		\State Define $P_{\cap}$ as
		\begin{equation}
		\begin{bmatrix} P_0 & P_1 & \cdots & P_{\ell} \end{bmatrix} := 
		\mathcal{T}  _L^T(P_a^T) Z_a =  \mathcal{T}_L^T(P_b^T) Z_b.
		\end{equation}
	\end{algorithmic}
\end{algorithm}
\color{black}
\section{Examples}
\label{sec:ex}

We show here some examples that illustrate the results derived in the previous sections. We propose simple analytical computations dealing with the addition and intersection of some LTI systems. 
\subsection{Scalar autonomous systems with simple poles}
\label{sec-4-1}
Consider two scalar autonomous  LTI systems $\mathcal{A}, \mathcal{B}$ defined by their minimal kernel representations 
\begin{equation*}
\mathcal{A} = \textrm{ker} \ R_a(\sigma) \quad\text{and}\quad \mathcal{B} = \textrm{ker} \ R_b(\sigma), 
\end{equation*}
where $R_a, R_b$ are scalar polynomials of degree $n_a = \mathbf{n}(\mathcal{A}), n_b = \mathbf{n}(\mathcal{B})$, respectively. Assuming that all the poles are simple, the trajectories of $\mathcal{A}$ and $\mathcal{B}$ are the sum of damped exponentials:
\begin{equation}
a = \sum_{i=1}^{n_a} a_i \lambda_{a_i}^t \quad\text{and}\quad b = \sum_{i=1}^{n_b} b_i \lambda_{b_i}^t,  \label{sumexp}
\end{equation}
for some coefficients $a_i, b_i$. By \eqref{sumexp} and the
definition of addition of behaviors \eqref{addbeh}, the trajectories
of the sum of two
LTI systems with simple poles are still sums of
damped exponentials, i.e., 
\begin{equation}
\label{eq:sum}
w_+(t) = \sum_{i=1}^{n_a} a_i \lambda_{a_i}^t + \sum_{i=1}^{n_b} b_i \lambda_{b_i}^t = \sum_{i=1}^{n_+} c_i \lambda_{+_i}^t,
\end{equation}
where the poles $\lambda_+$ are the union   of the poles of
$\mathcal{A}$ and $\mathcal{B}$: $\lambda_+ = \lambda(\mathcal{A} +
\mathcal{B}) = \lambda(\mathcal{A}) \cup \lambda(\mathcal{B})$. The
order $n_+$ is the number of distinct elements in $\lambda(\mathcal{A})
\cup \lambda(\mathcal{B})$, that is $n_a + n_b - n_c$ (where $n_c$ is the number of
common poles). 

By \eqref{sumexp} and the definition of intersection of behaviors \eqref{addbeh}, also the trajectories of the intersection of two LTI  behaviors contains sums of damped exponentials whose poles are the common poles of the two behaviors:
\begin{equation}
\label{eq:int}
w_{\cap}(t) = \sum_{i=1}^{n_a} a_i \lambda_{a_i}^t \cap \sum_{i=1}^{n_b} b_i \lambda_{b_i}^t = \sum_{i=1}^{n_{\cap}} d_i \lambda_{\cap_i}^t,
\end{equation}

 $\lambda_{\cap} = \lambda(\mathcal{A} \cap \mathcal{B}) = \lambda(\mathcal{A}) \cap \lambda(\mathcal{B})$. The order of $\mathcal{A} \cap \mathcal{B}$ is the number of common poles $n_c$ between $\mathcal{A}$ and $\mathcal{B}$. 

The previous results are summarized in the following lemma.
\begin{lemma}
	\label{lemma:lcmgcd}
	Let $\mathcal{A}$ and $\mathcal{B}$ be two scalar autonomous LTI behaviors with minimal kernel representations $R_a(\sigma), R_b(\sigma)$, defined by the scalar polynomials $R_a(z), R_b(z)$. The polynomials for the minimal kernel representations $R_+(z)$ of $\mathcal{A} + \mathcal{B}$ and $R_{\cap}(z)$ of $\mathcal{A} \cap \mathcal{B}$ are given by, respectively, the least common multiple and the greatest common divisor of $R_a(z)$ and $R_b(z)$. 
\end{lemma} 
\begin{proof}
	The expression of the trajectories of the sum and intersection systems are given in \eqref{eq:sum} and \eqref{eq:int}, respectively. The systems poles of the two starting systems are the exponents $\lambda_{a_i}, \lambda_{b_i}$, which are also roots of the (scalar) polynomials $R_a(z)$ and $R_b(z)$, respectively. 
	From \eqref{eq:sum}, we can see that the poles of the sum system are the union of $\lambda_{a_i}$ and $\lambda_{b_i}$, hence the roots of $R_+(z)$ are the union (without repetitions) of the roots of $R_a(z)$ and $R_b(z)$. Similarly, from \eqref{eq:int}, it follows that the roots of $R_{\cap}(z)$ are the intersection of the roots of $R_a(z)$ and $R_b(z)$. Since all the poles are simple, the thesis follows. 
\end{proof}
\begin{remark}
	The result of Lemma \ref{lemma:lcmgcd} can be naturally extended to the case of Multi-Input Multi-Output systems by replacing scalar with matrix polynomials. 
\end{remark}

The fact that $\lambda_+$ is the union of the poles of the two systems can also be checked in Matlab. Two (random) systems can be generated (in state-space form) by the function \textit{drss}, once we fix the number of inputs, outputs and the orders (in the input-output setting, the sum is well-defined only for systems with the same number of inputs and outputs). 
\begin{small}
\begin{verbatim}
 pole(sys1)      % poles first system
 pole(sys2)      % poles second system
 pole(sys1 + sys2) % poles of the sum
\end{verbatim}
\end{small} 
 While the addition of systems can be easily obtained by the sum, the intersection of two systems is not immediate to compute, and we should use some ad hoc algorithms (only  algorithms for the common dynamic estimation of scalar autonomous systems exist at the moment). Anyway, the Matlab function \textit{intersect} can be called to check the presence of common poles among the poles of the two systems.

\begin{remark}
  If the coefficients of the given representations are inexact, 
   the presence of common poles between the polynomials in the kernel representations could not be detected; they can be estimated
   by computing 
     approximate common divisors, e.g., via the algorithms developed in \cite{agcdode} for scalar polynomials (in the SISO case) or in \cite{matpol} for matrix polynomials (in the MIMO case). 
\end{remark}

\subsection{A single-input single-output system and an autonomous system}
\label{sec:ex2}

Consider a single-input single-output system and an autonomous system
with simple poles. The kernel representation of the first is given by
a $1 \times 2$ matrix polynomial:
\begin{equation*}
\mathcal{A} = \textrm{ker} \begin{bmatrix} q_a(\sigma) & p_a(\sigma) \end{bmatrix}.
\end{equation*}
The kernel representation of autonomous systems involves a square
matrix polynomial $R(\sigma)$ whose determinant is nonzero. The poles
are then the roots of the determinant of $R(\sigma)$. We consider  the
following kernel representation:
\begin{equation*}
\mathcal{B} = \textrm{ker} \begin{bmatrix} 1 & 0 \\ 0 & p_b(\sigma) \end{bmatrix}.
\end{equation*}
By choosing an input / output partition of the set of  variables\footnote{This partition is always possible if there are at least two
	variables. Starting from a difference equation of the form $R(\sigma) w
	= 0$, it is enough to switch to an input / output representation by
	splitting $w = (u; y)$ into a set of inputs and outputs and to partition $R = [Q, P]$
	accordingly.} $w = (u, y)$,
the trajectories of the first system satisfy the equation 
\begin{equation*}
q_a(\sigma)
u = -p_a(\sigma) y \quad\iff\quad y = -q_a(\sigma) / p_a(\sigma) u = h_a * u,
\end{equation*}
 where the star denotes the convolution product. 
We need to add to these trajectories the free response $y_{a, f} \in
\text{ker}\, p_a(\sigma)$,  which are the trajectories corresponding to
zero input. Hence, the trajectories of the system $\mathcal{A}$ have
the general form
\begin{equation*}
w_a = \begin{bmatrix} u \\ y_{a, f} + h_a * u \end{bmatrix}. 
\end{equation*}  
Observe that both the free response $y_{a, f}$ as well as the impulse
response $h_a$ are sums of damped exponential signals of the form
$\sum_{i=1}^{n_a} \alpha_i z_{a, i}^t$ where $n_a$ is the degree of
the polynomial $p_a$, while $z_{a, 1}, \dots, z_{a, n_a}$ are the
roots of $p_a$. 

The trajectories of the system $\mathcal{B}$ satisfy the equation
\(
\bsm u \\ p_b(\sigma) y \esm = 0.
\)
 We see that the input can only be zero so that the output is constrained to the free response, i.e., $y_{b, f} \in \text{ker}\, p_b(\sigma)$. The trajectories of the system $\mathcal{B}$ have the general form 
\(
w_b = \bsm 0 \\ y_{b, f} \esm.
\)
The free response $y_{b, f}$ is still a sum of damped exponentials
 whose exponents are the poles of the system $\mathcal{B}$. 

The sum $\mathcal{A} + \mathcal{B}$ is a single-input single-output system whose
trajectories  have the form
$w_+ = w_a + w_b$, and the poles $\lambda(\mathcal{A} +
\mathcal{B})$ are the union of the poles $\lambda(\mathcal{A}) \cup \lambda(\mathcal{B})$.  But the poles
$\lambda(\mathcal{B})$ appear only in the free response and not in the
convolution with the input. A kernel representation of the sum
$\mathcal{A} + \mathcal{B}$ is given by 
\begin{equation}
\label{eq:kersum}
R_+(z) = p_b(z) \begin{bmatrix} q_a(z)  & -p_a(z) \end{bmatrix}. 
\end{equation}  
The kernel representation \eqref{eq:kersum} shows that the sum of a single-input single-output system with an autonomous system 
 is always uncontrollable because of the presence of the common
factor $p_b(z)$ \cite{PW98} (i.e., the matrix $R_+(z)$ is not left prime). 

The trajectories of the intersection $\mathcal{A} \cap \mathcal{B}$ should be of the form $w_a$ and $w_b$ at the same time. Hence, the input is constrained to be zero and the output contains only the free response $y_{\cap, f}$, which should be in the kernels of both $p_a(z)$ and $p_b(z)$, i.e., in the kernel of their greatest common divisor. Hence, the intersection is an autonomous system  whose kernel representation has the following expression:
\begin{equation}
\label{eq:keraut}
R_{\cap}(z) = \begin{bmatrix} 1 & 0 \\ 0 & \text{gcd}\ \big(p_a(z), p_b(z)\big) \end{bmatrix}. 
\end{equation}

\section{Conclusion}
We studied  the two basic operations of addition and intersection of LTI systems in the behavioral setting, showing that they are different from the classical definitions in the input-output setting. 
The proposed trajectory-based definition of sum and intersection allows us to perform such computations directly from the data. 
Moreover, we saw how the resulting system  representations depend on the representations of the starting systems, and we proposed algorithms for their computations based on structured matrices and polynomial computations. We summarize some of the main advantages due to  the proposed definitions and results:
\begin{enumerate}
	\item the intersection has been extended to open  systems (systems with inputs);
	
	\item the two operations can be performed directly using the system trajectories (observed data); the system representations, if needed for the problem, can be computed at a later stage;
	
	\item we can sum and intersect systems with different numbers of inputs and outputs (but the same number of variables!).
\end{enumerate}

The development and implementation of an algorithm to estimate the common dynamics among open systems can be object of future research. 
		
\section*{Acknowledgment}

 
Ivan Markovsky is an ICREA research professor. The research leading to these results has received funding from: the Catalan Institution for Research and Advanced Studies (ICREA), the Fond for Scientific Research Vlaanderen (FWO) projects G090117N and G033822N; and the Fonds de la Recherche Scientifique FNRS--FWO EOS Project 30468160. Antonio Fazzi was supported by the Italian Ministry of University and Research under the PRIN17 project Data-driven learning of constrained control systems, contract no. 2017J89ARP.

\section*{Conflict of interest}
The authors declare that they have no conflict of interest.

\section*{References}

\bibliographystyle{elsarticle-num}
\bibliography{mypapers,add}

\end{document}